\documentclass{amsart}

\usepackage{amssymb}
\usepackage{mathrsfs}
\usepackage{amsmath}

\newtheorem{Th}{Theorem}[section]
\newtheorem{Cor}[Th]{Corollary}
\newtheorem{Lem}[Th]{Lemma}
\newtheorem{Prop}[Th]{Proposition}

\newtheorem{Rem}[Th]{Remark}

\newtheorem{claim-num}{Claim}

\renewcommand{\le}{\leqslant}
\renewcommand{\ge}{\geqslant}

\def\N{\mathbf N}
\def\Z{\mathbf Z}
\def\aut#1{\operatorname{Aut}(#1)}
\def\rank{\operatorname{rank}}
\def\f{\varphi}
\def\s{\sigma}
\def\Mod#1{\,(\operatorname{mod} #1)}
\def\id{\mbox{\rm id}}
\def\vk{\varkappa}
\def\inv{^{-1}}
\def\str#1{\langle #1 \rangle}
\def\To{\Rightarrow}
\def\sle{\subseteq}

\def\nc{\mathrm{nc}}

\def\cB{\mathscr B}
\def\cC{\mathscr C}

\def\cF{\mathcal F}

\def\cT{\mathscr T}
\def\cU{\mathscr U}

\def\cX{\mathscr X}
\def\cY{\mathscr Y}

\newcounter{corrno}
\numberwithin{equation}{section}

\def\btau{\boldsymbol{\tau}}
\def\aR{\mathrm{aR}}

\def\btau{\boldsymbol{\tau}}
\def\aR{\mathrm{aR}}

\def\GammamA{\Gamma_{\!A}(m)}
\def\LambdamA{\Lambda_{\!A}(m)}

\def\GGamma#1{\Gamma_{\!A}(#1)}
\def\LLambda#1{\Lambda_{\!A}(#1)}

\title[Principal congruence subgroups]
 {Principal congruence subgroups in the infinite rank case}

\author{Vladimir A. Tolstykh}

\address{Vladimir A. Tolstykh\\ Department of Mathematics and Computer Science \\ Istanbul Arel University \\ 34537 Tepekent - B\"uy\"uk\c{c}ekmece \\ Istanbul \\ Turkey}
\email{vladimirtolstykh@arel.edu.tr}

\subjclass{20K30 (primary), 20H05, 20F28 (secondary)}

\begin{document}

\maketitle

\begin{abstract}
We obtain a number of analogues of the classical
results of the 1960s on the general
linear groups $\mathrm{GL}_n(\Z)$
and special linear groups $\mathrm{SL}_n(\Z)$ for the
automorphism group $\Gamma_A=\aut A$
of an infinitely generated
free abelian group $A.$
In particular, we obtain
a description of normal generators of the
group $\aut A,$ classify the maximal
normal subgroups of the group $\aut A,$
describe normal generators of the
principal congruence subgroups $\GGamma m$
of the group $\aut A,$ and obtain
an analogue of Brenner's ladder
relation for the group $\aut A.$
\end{abstract}


\section*{Introduction}

In \cite{Brenner} Brenner initiated a program of
study of the structure of the normal subgroups of the general
linear groups $\mathrm{GL}_n(\Z)$ and special
linear groups $\mathrm{SL}_n(\Z).$
In the case when $n \ge 3$ he proved
the following result  \cite[Th. 2]{Brenner}: letting $\Gamma_n$ denote $\mathrm{GL}_n(\Z),$
or $\mathrm{SL}_n(\Z),$ every normal subgroup
$N$ of $\Gamma_n$ is sandwiched
between the normal closure $\Gamma^*_n(m)$ of the elementary
transvection $t_{12}(m) =\mathrm I + m \mathrm I_{12}$
in $\Gamma_n$ and the normal subgroup $\Lambda_n(m)$ of $\Gamma_n$ for a suitable natural
number $m \ge 0,$
$$
\Gamma^*_n(m) \le N \le \Lambda_n(m),
$$
thereby satisfying a relation called the {\it ladder}
relation in \cite{Brenner}. Here $\Lambda_n(m)$ is the subgroup of all matrices
of $\Gamma_n$ that are congruent to a scalar matrix modulo $m.$

Further, it has been conjectured in \cite{Brenner} that in fact
$$
\Gamma^*_n(m)=\Gamma_n(m) \text{ in } \Gamma_n=\mathrm{SL}_n(\Z),
$$
that is, the principal congruence subgroup $\Gamma_n(m)$ of $\mathrm{SL}_n(\Z)$
of level $m,$ consisting of all matrices congruent
to the identity matrix modulo $m$, is equal to the normal closure of the
transvection $t_{12}(m).$ Brenner himself made a significant
contribution towards the proof of the conjecture,
which has been confirmed later by Mennicke \cite{Mennicke}
and by Bass, Lazard, and Serre \cite{BLS}.
As a corollary, we obtain, as it has been also
suggested in \cite{Brenner}, that every subgroup of $\mathrm{SL}_n(\Z)$
of finite index contains the principal
congruence subgroup of some level $m > 0$ (the congruence
subgroup property).

Observe that if $p$ is a prime number,
then the normal subgroup $\Lambda_n(p)$ is
a maximal normal subgroup of the group $\mathrm{SL}_n(\Z)$
and is a lifting of the only maximal normal
subgroup of the special linear
group $\mathrm{SL}_n(\Z/p\Z).$

The automorphism group $\aut A$ of an infinitely generated
free abelian group $A$ naturally takes over both the group $\mathrm{GL}_n(\Z)$
and the group $\mathrm{SL}_n(\Z),$ since
there is no determinant-like function
on $\aut A,$ because it is perfect \cite[Th. 1.5, Prop. 2.1]{To_Berg}.
The aim of the present paper is to obtain
analogues of the above quoted results for
the normal subgroups of the group $\aut A.$
To the best of this author's knowledge,
there are no publications in which
a full-scope program of study of the
structure of the normal subgroups of the group
$\aut A$ similar to Brenner's program has
been addressed (though there is a number of results
on the normal subgroups of finitary
automorphisms in the automorphism groups
of free modules of infinite rank, see,
for instance, \cite{Arrell,Farouqi}).

Let $M$ be a free module of infinite
rank over a ring. Given a
decomposition $M=S \oplus L$
into a direct sum of free submodules,
we call the free submodule $L$ {\it large} (resp. {\it moietous})
if $\rank(S) < \rank(M)$ (resp. $\rank(S)=\rank(L)).$
An automorphism $\f \in \aut M$ is called
{\it almost a radiation}
if $\f$ acts as a radiation
on a large free direct summand of $M.$ According
to Rosenberg \cite[Th. B]{Rosenberg}, if $V$ is
an infinite-dimensional vector space
over a division ring, then the only maximal
normal subgroup of the general linear
group $\mathrm{GL}(V)=\aut V$ of $V$
is the subgroup of all almost-radiations.

This result, much like the observation
we have made above, leads to the following
definition: given a natural number $m \ge 0,$
we define the subgroup $\LambdamA$ of the group $\aut A$ to be the
subgroup of all automorphisms $\f \in \aut A$
for which there is an integer number $k$
and a large direct summand $L$ of $A$ such that
$$
\f(a) \equiv ka \Mod{mA}
$$
for all $a \in L.$ In the case when $m=0,$
the group $\LLambda 0$ is equal the group $\aR(A)$
of all almost-radiations from $\aut A.$
Observe also that each almost-radiation belongs to all
subgroups $\LLambda s,$ where $s \in \N.$ In the case when $m=p$ is a prime
number, we obtain that the group $\LLambda p$
is a maximal normal subgroup of the group
$\aut A,$ since the homomorphism
$\aut A \to \aut{A/pA}$ induced by
the natural homomorphism $A \to A/pA$
is onto \cite{BrMa,BurnsFa}. A natural question to ask
is whether the converse is true,
that is, whether, analogously
to the case of the groups $\mathrm{SL}_n(\Z),$
the subgroups $\LLambda p$ are exactly
all maximal normal subgroups of the group $\aut A.$

In turn, the
principal congruence subgroup $\GGamma m$
of $\aut A$ of level $m$ is defined
in a straightforward fashion as the family
of all automorphisms $\f$ of $A$ such that
$$
\f(a) \equiv a \Mod{mA}
$$
for all $a \in A,$ or equivalently as
the kernel of the homomorphism
$\aut A \to \aut{A/mA}$ induced
by the natural homomorphism $A \to A/mA.$

The paper is organized as follows. In the
first section we describe normal generators
of the group $\aut A.$ We prove that
an automorphism $\f$ is a normal
generator of the group $\aut A$
if and only if there is a moietous direct summand
$W$ of $A$ such that
$$
A = W \oplus \f(W) \oplus V
$$
for a suitable direct summand $V$ (Theorem \ref{normal_generators_both}).
In particular, if $(x_i,y_i : i \in I)$ is a basis of the
group $A,$ then the automorphism $\tau \in \aut A$ defined
by
\begin{align*}
\tau x_i &= x_i +y_i, \qquad (i \in I), \\
\tau y_i &=y_i
\end{align*}
is a normal generator of the group $\aut A,$
naturally taking over the normal generator
$t_{12}(1)$ of the group $\mathrm{SL}_n(\Z).$
In the process of the proof of
Theorem \ref{normal_generators_both}, we demonstrate that if $\f \in \aut A$
is not a normal generator, then $\f \in \LLambda m$
for some natural number $m \ge 2.$

In the second section we work to classify the maximal
normal subgroups of the group $\aut A.$ Let $P$
be a subset of the set $\mathbf P$ of prime
numbers. Write $\Omega_P$ for the normal
subgroup $\bigcap_{p \in P} \LLambda p.$
Most generally, the maximal normal
subgroups of $\aut A$ are the normal
subgroups of the form
$$
\LLambda \cF =\bigcup_{P \in \cF} \Omega_P,
$$
where $\cF$ is an ultrafilter over $\mathbf P.$
More precisely (Theorem \ref{max_norm_subgrs}), in
the case when the cardinal $\rank(A)$
is of {\it uncountable} cofinality,
and quite similar to the situation
with the groups $\mathrm{SL}_n(\Z)$
we have outlined above,
the normal subgroups $\LLambda p,$
where $p$ ranges over $\mathbf P,$ are exactly the maximal
normal subgroups of the group $\aut A.$ As
a matter of fact, if $\rank(A)$ is of uncountable
cofinality and if $P$ is an infinite set of primes, then the group $\Omega_P$ is equal
to the group $\aR(A)$ of all almost-radiations (Proposition \ref{uncount_conf_no_inf_descent}),
which implies that $\LLambda \cF=\aR(A)$
for all nonprincipal ultrafilters $\cF$
over $\mathbf P.$ In the case when the cardinal
$\rank(A)$ has countable cofinality, the
normal subgroups $\LLambda \cF,$ where
$\cF$ runs over the family of all ultrafilters
over $\mathbf P,$ are pairwise distinct,
and are exactly the maximal normal
subgroups of the group $\aut A.$

We begin Section 3 by proving that for every
natural number $m \ge 0$ the normal closure of the $m$-th power
$\tau^m$ of the automorphism $\tau,$ defined above, in the group $\aut A$ is equal to
the principal congruence subgroup $\GGamma m$
(Proposition \ref{tau^m}).
We then prove a sufficient condition
for an automorphism $\f \in \aut A$
to contain the principal congruence
subgroup of some level $m \ge 2$ in
its normal closure (Proposition \ref{norm_gens_of_Km}).
Next, we show that if $N$ is a proper normal subgroup
of $\aut A$ having finitely many normal
generators and
if there exists a maximal natural number $m$
satisfying $N \le \LLambda m,$ then $\GGamma m \le N.$
This result is best possible in the case when
$\rank(A)$ is of countable cofinality.
On the other hand, if the cardinal $\rank(A)$
has uncountable cofinality,
then every normal subgroup
$N$ of $\aut A$ satisfies
Brenner's ladder relation,
that is, to repeat,
there exists a natural number $m \ge 0$
satisfying
$$
\GGamma m \le N \le \LLambda m
$$
(Proposition \ref{LadderRel}).

As we have seen above, a number of important
properties of the groups $\mathrm{GL}_n(\Z)$
and $\mathrm{SL}_n(\Z)$ admit rather
straightforward analogues for the
automorphism groups of infinitely
generated free abelian groups.
However, it is not the case for
the congruence subgroup property for
the groups $\mathrm{SL}_n(\Z),$ where $n \ge 3.$
Indeed, imposing any condition
on a subgroup $H$ of the group $\aut A$
in terms of the index of $H$
in $\aut A$ alone does not imply
that $H$ contains the
principal congruence
subgroup $\GGamma m$
for a suitable $m \ge 2.$
For instance, the stabilizer $H$ of a nonzero
element of $A$ in the group $\aut A,$ a subgroup
of index $\rank(A),$ does not contain nonidentity
normal subgroups of $\aut A,$ since evidently
$$
\bigcap_{\sigma \in \aut A} H^\sigma = \{\id_A\}.
$$
Speaking of the subgroups of $\aut A$ of strictly smaller index,
the group $\aut A$ does not have any under, say the GCH, since
for every cardinal $\nu$ satisfying
$2^\nu < 2^{\rank(A)},$ the group $\aut A$
has no subgroups of index $\nu$ \cite[Prop. 2.2]{To_Smallness4Nilps}.

The author would like to express his gratitude to George Bergman, Simon Thomas,
and Vitali\u{\i} Roman'kov for useful information, and to thank
the referee for helpful comments.

\section{Normal generators}

Everywhere below $A$ stands for an infinitely
generated free abelian group which we shall
fix till the end of the paper. We shall
use the letter $\vk$ to denote
the cardinal $\rank(A).$ The automorphism
group $\aut A$ of $A$ will be denoted by
$\Gamma.$ We split the family of all direct
summands of $A$ into three subfamilies:
the subfamilies of small, moietous,
and large direct summands. Let $R$
be any direct summand of $A.$
We call $R$
\begin{itemize}
\item {\it small} if rank of $R$ is strictly less than
the cardinal $\vk=\rank(A);$
\item {\it moietous} if both rank
and corank of $R$ are equal to $\vk$ \cite{Macph,To_Berg};
\item {\it large} if corank of $R$
is strictly less than $\vk.$
\end{itemize}
Observe that the sum of two small direct
summands is contained in a small direct
summand, that the intersection
of two large direct summands
is again a large direct summand,
and that the intersection of a large direct summand
and a moietous direct summand is a moietous
direct summand.

Recall that an element of $A$ which can
be included into a basis of $A$ is called
{\it unimodular}. It is also convenient to call
any nonempty subset of $A$ which can be
extended to a basis of $A$ a {\it unimodular set} of $A.$ Clearly,
a nonempty subset of $A$ is unimodular
iff it is a basis of a nonzero direct summand of $A.$

\begin{Prop} \label{Normalgenerators_suff_part}
Suppose an automorphism $\f \in \aut A$
has a moietous direct summand $W$ of $A$ such that
$$
A = W \oplus \f W \oplus V
$$
for some direct summand $V$ of $A.$ Then $\f$ is
a normal generator of the group $\aut A.$
\end{Prop}

\begin{proof}
Recall that according to the well-known description of normal
subgroups of the symmetric group $\text{Sym}(I)$ of an infinite
set $I$ \cite[Sect. 8.1]{DiMo}, the proper normal subgroups
of $\mathrm{Sym}(I)$ are exactly the normal subgroups
$$
\{\id\}, \quad \mathrm{Alt}(I),\quad \text{Sym}_\lambda(I),
$$
where $\mathrm{Alt}(I)$ is the alternating group
of $I,$ $\lambda$ is an arbitrary infinite cardinal $\le |I|$
and $\mathrm{Sym}_\lambda(I)$ consists of all permutations
of $I$ whose support is of cardinality $< \lambda.$
Consequently, any permutation of $I$ whose support is of cardinality
$|I|$ is a normal generator of the group $\text{Sym}(I).$

Next, we are going to describe
certain automorphisms of $A$ that are known to be normal generators
of $\Gamma=\aut A.$ According to Proposition 2.1
and the proof of part (i) of Theorem 1.5 from \cite{To_Berg},
any involution of $A$ which acts on a suitable basis
of $A$ as a permutation of order two whose support
and the fixed-point set are both of cardinality $\rank(A)$ is a normal
generator of $\Gamma.$
It follows that any automorphism
of $A$ which acts on some basis of $A$
as a permutation whose support is of
cardinality $\rank(A)$ is a normal generator
of $\Gamma.$

Let $(x_i : i \in I)$ and $(x_i' : i \in I)$ be bases of the direct
summand $W$ and let $(v_j : j \in J)$ be a basis
of the direct summand $V.$ As, by the conditions,
$$
(x_i : i \in I) \sqcup (\f x_i : i \in I) \sqcup (v_j : j \in J),
$$
where $\sqcup$ denotes the disjoint union
of sets, is basis of $A,$ there is an involution $\rho \in \Gamma$
which interchanges $\f x_i$ and $x_i'$
for all $i \in I$ and takes each element $v_j$
to itself for all $j \in J.$ Accordingly,
$\rho \f W = W,$ and hence
$$
\rho \f \rho \f W =W.
$$
Furthermore, the automorphism $\f_1 = \rho \f \rho \f$
is in the normal closure of $\f$ and, due to arbitrariness
of $(x_i' : i \in I),$ the restriction
of $\f_1$ on $W$ could be equal to
the square of any automorphism of $W.$
We can therefore assume that $\f_1$
acts on the basis $(x_i : i \in I)$
as a permutation with $\rank(A)$
infinite cycles and the fixed-point
set of cardinality $\rank(A).$

Now let $U=\f(W) \oplus V$ and let $(u_j + y_j : j \in J)$
be a basis of the subgroup $\f_1 U,$
where $u_j \in U$ and $y_j \in W$ for all $j \in J.$
Clearly, $(u_j : j \in J)$ is a basis
of the group $U.$ Consider the involution $\rho_1 \in \Gamma$
which sends every element of $W$ to the
opposite one, and for which we have that
$$
\rho_1(u_j)=u_j+y_j
$$
for all $j \in J.$ It follows $\rho_1 \f_1 U =U,$
and hence the automorphism $\f_2 = \rho_1 \f_1 \rho_1 \f_1$
preserves both $W$ and $U,$ and the action
of $\f_2$ on the basis $\cX=(x_i : i \in I)$ is isomorphic
to that one of $\f_1.$ Let $W_1$
be the span of the support of the permutation
$\f_2|_{\cX}$ of the basis $\cX$ and $W_2$ the span of the
fixed part of this permutation; both
$W_1$ and $W_2$ are moietous direct
summands by the construction. The
automorphism $\f_2$ can then be described
as the direct sum
$$
\psi \oplus \id \oplus \s
$$
taken over the decomposition
$$
A=W_1 \oplus W_2 \oplus U,
$$
where $\psi$ (resp. $\id,\s$) refers to the (isomorphism type of the) restriction
of $\f_2$ on $W_1$ (resp. on $W_2,U$). Accordingly,
the automorphism
$$
\id \oplus \psi\inv \oplus \s\inv
$$
taken over the same decomposition
is a conjugate of $\f_2\inv,$ and their
product is
$$
\f_3 = \psi \oplus \psi\inv \oplus \id.
$$
Now the automorphism $\f_3,$ an element
of the normal closure of $\f,$ acts on a suitable basis of $A$
as a permutation with the support
of cardinality $\rank(A),$ and
hence, as we remarked above,
is a normal generator of the
group $\Gamma.$
\end{proof}

Next, we describe a normal generator of the group $\Gamma=\aut A$
that naturally takes over the standard normal generator
$$
t_{12}(1) = {\mathrm I} + {\mathrm I}_{12}
$$
of the group $\mathrm{SL}_n(\Z),$ where $n \ge 2,$ ${\mathrm I}$ is the identity matrix of size
$n \times n$ and ${\mathrm I}_{12}$ is the matrix unit of size $n \times n$ (all of whose entries,
except for the $(1,2)$-entry which is equal to $1,$ are zeros).

\begin{Cor}
Let $(x_i,y_i : i \in I)$ be a basis of the group
$A.$ Then the normal closure of the automorphism $\tau \in \aut A$ defined
by
\begin{align*}
\tau x_i &= x_i +y_i, \qquad (i \in I), \\
\tau y_i &=y_i
\end{align*}
is equal to the group $\aut A.$
\end{Cor}

\begin{proof}
We have that
$$
A = \str{x_i : i \in I} \oplus \tau\str{x_i : i \in I},
$$
whence the result.
\end{proof}

We will denote the isomorphism type of the
action of $\tau$ on $A$ by $\btau.$ We shall see later that, quite similar
to the situation with the groups $\mathrm{SL}_n(\Z),$
given a natural number $m \ge 0,$ the normal closure of
any automorphism of type $\btau^m$ is equal
to the principal congruence subgroup $\GammamA$
of $\Gamma.$

Let $m \ge 0$ be a natural number. Then the {\it principal
congruence subgroup} $\GammamA$ of $\Gamma=\aut A$ of level $m$ consists
of all automorphisms $\f$ of $A$ such that
$$
\f a \equiv a \Mod{mA},
$$
or, simply put,
$$
\f a \equiv a \Mod m
$$
for all $a \in A.$ Equivalently, $\GammamA$
is the kernel of the homomorphism $\aut A \to \aut{A/mA}$
induced by the natural homomorphism $A \to A/mA.$
We define the subgroup $\LambdamA$ of $\Gamma=\aut A$
to be the subgroup of all automorphisms $\f \in \Gamma$
such that there are a large direct summand $M$ of
$A$ and an integer number $k$ satisfying
$$
\f x \equiv k x \Mod m
$$
for all $x \in M.$ Geometrically, an element $\f \in \LambdamA$
induces an automorphism of the group $A/mA$
which is `almost a radiation', for it acts
as a radiation on a large direct summand
of this group.

\begin{Prop} \label{Nongenerators}
Let $\f \in \aut A$ be such that
there is no moietous direct summand $W$ of $A$
having the property that $W \oplus \f W$
is a direct summand of $A.$
Then $\f \in \LambdamA$ for a suitable
natural number $m \ge 2.$ In particular,
$\f$ is not a normal generator
of the group $\aut A.$
\end{Prop}

\begin{proof} Given a direct summand $B$
of $A$ and a unimodular element $u \in A,$
we will write
$$
B \to u
$$
iff $B \oplus \str u$ is a direct summand
of $A.$ In other words, $\cB \cup \{u\}$
is a unimodular set of $A$ for every basis $\cB$
of $B.$ Consequently, we have that
$$
B \not\to u
$$
iff there are an element $b \in B$ and a natural number $m \ge 2$
such that
$$
u \in b + mA;
$$
equivalently, the image of $u$ is in the image of $B$
under the natural homomorphism $A \to A/mA.$

\begin{claim-num}
There is a small direct summand $X$ of $A$ {\rm(}of rank $< \vk${\rm)} such that
$$
X \to t \To \str{X,t} \not\to \f(t)
$$
for every unimodular element $t \in A.$
\end{claim-num}

Indeed, suppose otherwise. Then for
every small direct summand $Y$ of $A$
there is a unimodular element $t \in A$
such that
$$
Y \to t \text{ and } \str{Y,t} \to \f(t).
$$
Let then
$
\cB=(b_i : i < \vk)
$
be a basis of $A$ indexed by the elements of the cardinal $\vk$
(here and below we always assume that any ordinal
which is strictly less than $\vk$ is of strictly
smaller cardinality). Using the transfinite induction,
we are going to construct a chain $(\cC_i : i < \vk)$
of unimodular sets of $A$ of cardinality $< \vk$ such that
\begin{itemize}
\item the direct summand of $A$ generated
by any member $\cC_i$ of the chain is equal
to the direct summand generated by a suitable
subset $\cB_i$ of the basis $\cB;$
\item the union of the chain $(\cC_i : i < \vk)$
is a basis of $A$
\end{itemize}
and a moiety $\cT$ of $\cC$ such that
\begin{equation} \label{str(T),f(str(T))}
\str\cT \oplus \f(\str\cT)
\text{ is a direct summand of $A$}
\end{equation}
Take $\cC_0$ (resp. $\cT_0$) to be the empty set. Suppose that
$i = j+1$ is an element of $\vk.$ Consider
the direct summand $\str{\cC_j}$ and let
$t_i \in A$ be a unimodular element of $A$
satisfying
\begin{equation}
\str{\cC_j} \to t_i \text{ and } \str{\cC_j,t_i} \to \f(t_i).
\end{equation}
By the construction,
$$
\str{\cC_j}=\str{\cB_j}
$$
for some subset $\cB_j$ of $\cB.$
Now we add to $\cB_j$ the minimal
finite subset $\cB_j'$ of $\cB$
such that
$$
b_i,t_i,\f(t_i) \in \str{\cB_j \cup \cB_j'}.
$$
By (\theequation), there is a basis
of the subgroup
$$
\str{\cB_j \cup \cB_j'}
$$
which contains $\cC_j$ and both $t_i$
and $\f(t_i),$ and we take this basis
to be our next unimodular set $\cC_i.$ Further,
$$
\cT_i=\cT_j \cup \{t_i\}.
$$
When $i < \vk$ is a limit ordinal, we set
$$
\cC_i =\bigcup_{j < i} \cC_j=\bigcup_{j < i} (\cC_j \setminus \bigcup_{k < j} \cC_k) \text{ and } \cT_i =\bigcup_{j < i} \cT_j;
$$
the second equality demonstrates that $\cC_i$ is a union
of $|i|$ of finite sets, and hence is of cardinality
at most $|i| \cdot \aleph_0= |i| < \vk.$ Setting then
$$
\cC=\bigcup_{i < \vk} \cC_i,
$$
we obtain that
$$
\cT=\bigcup_{i < \vk} \cT_i
$$
satisfies \eqref{str(T),f(str(T))}, which contradicts
the hypothesis of the Proposition.

\begin{claim-num}
There is a large direct summand $T \le A$ {\rm(}of corank $< \vk${\rm)}
and coprime natural numbers $k,m,$ where $m \ge 2$
such that
$$
\f(t) \equiv kt \Mod m
$$
for all $t \in T.$
\end{claim-num}

According to Claim 1, there exists a small direct summand $X$
such that
$$
X \to t \To \str{X,t} \not\to \f(t)
$$
for all unimodular elements $t \in A.$ In effect,
assuming that $X \to t$ for a unimodular
element $t \in A,$ we obtain that $\str{X,t} \not\to \f(t),$
and hence
$$
\f(t) = x + kt + mz
$$
where $x=x(t) \in X,$ $k=k(t)\in \Z,$ $m=m(t) \in \N \setminus \{1\},$
and $z=z(t) \in A$ is an unimodular
element for which have that
$$
\str{X,t} \to z.
$$

Consider
a direct complement $T_0$ of $X$ to $A,$ and let
$$
T_1 = T_0 \cap \f\inv( T_0).
$$
The direct summand $T_1,$ being an intersection
of two large direct summands is itself
large. By the construction, $\f(T_1) \le T_0,$
and hence for every unimodular element $t \in T_1,$
$$
\f(t) = kt + mz
$$
where $k=k(t) \in \Z$ and $m=m(t) \in \N \setminus \{1\},$
and $z=z(t)$ is a unimodular element of $T_0$ such that
$\str{X,t} \to z;$ observe also that the number $m(t)$
is uniquely determined (whereas $k(t)$ and $z(t)$ are not).

Now if $m(t)=0$ for all unimodular elements $t$ in $T_1,$ we
get that $\f(t)=k(t) t,$ where $k(t) \in \Z,$ or, obviously,
$\f(t) = \pm t$ for {\it all} unimodular elements $t$ of $T_1.$ This easily implies
that $\f$ either fixes each of element $T_1,$ or takes
each element of $T_1$ to the opposite one, whence $\f \in \LambdamA$
for all $m \ge 2,$ and we are done.

Suppose that the function $m(t)$ does not vanish on the family $\cU(T_1)$
of all unimodular elements of $T_1$ and let $t_0 \in T_1$
be a unimodular element such that
$$
m_0=m(t_0)
$$
is the minimal positive value of the function $m(t)$
on $\cU(T_1).$ Let further
$$
\f(t_0)=k_0 t_0 + m_0 z_0
$$
where $\{t_0,z_0\}$ is a unimodular pair in $T_0.$
Write $T_0$ as
$$
T_0 = \str{t_0,z_0} \oplus T_0'
$$
and set
$$
T_2= T_1 \cap T_0'.
$$
and
$$
T_3 = T_2 \cap \f\inv(T_2).
$$
Again, both $T_2$ and $T_3$ are large direct summands, and
$$
\f(T_3) \le T_2.
$$

Suppose that $t \in T_3$ is a unimodular element and
let
$$
\f(t)=k t+m z
$$
where $\{t,z\}$ is a unimodular pair in $T_2.$ We claim
that $m_0$ divides $m.$ Importantly,
$$
\{t_0,z_0,t,z\}
$$
is a unimodular set of $T_0.$ Divide $m$ by $m_0$ with remainder,
$$
m = q m_0 + r
$$
where $0 \le r < m_0$ and consider the action
of $\f$ on the element $t - q t_0 \in T_1$:
\begin{align*}
\f(t-q t_0) &= k t + (m_0 q+r) z - (k_0q t_0 + m_0qz_0) \\
	    &= k (t-q t_0) + (k-k_0)q t_0 +m_0 q( z-z_0) + r z.
\end{align*}
Since
$$
\{t-qt_0, t_0, z-z_0, z\}
$$
is a unimodular set of $T_0,$ we deduce that $r=0,$
for otherwise $m(t-qt_0) \le r$ would be less than $m(t_0).$

Next, we prove that $k-k_0$ is divisible
by $m_0.$ Indeed, let
$$
k-k_0= q' m_0+r'
$$
where $0 \le r' < m_0$ and consider the
action of $\f$ on $t-t_0 \in T_1$:
\begin{align*}
\f(t-t_0) &= k t+mz - k_0 t_0 -m_0 z_0\\
	  &= k(t-t_0) + (k-k_0)t_0 -m_0 z_0 +m z\\
	  &= k(t-t_0) + (q' m_0 +r')t_0 -m_0 z_0 + mz\\
	  &= k(t-t_0) + r' t_0 +m_0 (q't_0-z_0) +mz.
\end{align*}
Again, since
$$
\{t-t_0, t_0, q't_0 -z_0,z\}
$$
is a unimodular set of $T_0,$ we would get
a contradiction to the choice of $m(t_0)$
if we had $r' \ne 0.$
\end{proof}

\begin{Th} \label{normal_generators_both}
An automorphism $\f$ of the group $A$ is
a normal generator of the group $\aut A$
if and only if there is a moietous direct
summand $W$ of $A$ such that $W \oplus \f(W)$
is a direct summand of $A.$
\end{Th}

\begin{proof}
By Proposition \ref{Normalgenerators_suff_part}
and Proposition \ref{Nongenerators}.
\end{proof}

\begin{Rem} \label{BrTrIndArg}
\rm In what follows, much as in the proof of Claim 1
in the proof of Proposition \ref{Nongenerators}, we will need on a couple of occasions to construct
a moiety $\cT$ of a basis $\cC$ of $A$ with a number of prescribed properties.
Each time we will use a modified version of the proof of the aforementioned claim:
namely, given any small direct summand $X$ of $A$ with a basis
$\cX,$ we will first demonstrate that it is possible to find a unimodular element $t$
of $A,$ a new potential member $\cT,$ possibly along with a finite number $z_1,\ldots,z_k$ of unimodular
elements of $A$ such that $\cX \cup \{t\} \cup \{z_1,\ldots,z_k\}$ is a unimodular set of $A,$
to ensure that $t$ behaves as required. Then existence of $\cC$ and $\cT$ will become a simple
matter of application of the transfinite induction
in the spirit of the transfinite induction
argument from the proof of the Claim. $\triangle$
\end{Rem}

\section{Maximal normal subgroups}

We will call an automorphism $\f \in \aut A$
{\it almost a radiation} if there is a large direct
summand $M$ of $A$ such that $\f$ acts on $M$ either as $\id_M,$ or
as $-\id_M$ (we have already met these
automorphisms in the proof of Proposition \ref{Nongenerators}). It is easy to see that
the family $\aR(A)$ of all almost-radiations
of $A$ is a normal subgroup of $A.$

\begin{Prop} \label{Almost-Radiations}
Let $\f$ be an automorphism of the group $A.$ Then the following
are equivalent:

{\rm (i)} $\f$ is an almost-radiation;

{\rm (ii)} for every moietous direct summand $U$ of $A,$ the subgroup $\f(U) \cap U$ is a
large direct summand of both the group $U$
and the group $\f(U);$

{\rm (iii)} for every moietous direct summand $W$ of $A,$
$\f(W) \cap W \ne \{0\}.$
\end{Prop}

\begin{proof} (i) $\To$ (ii). Suppose that $\f \in \Gamma$ is an almost-radiation
and let $M$ be a large direct summand on which $\f$ acts either
as $\id_M,$ or as $-\id_M.$ Now if $U \sle A$ is any moietous
direct summand, $U_0=U \cap M$ is a large direct summand of the group $U,$ and since evidently
$$
\f(U_0)=\f(U \cap M)=U\cap M=U_0,
$$
the subgroup $U_0$ is a large direct summand of the group $\f(U).$

(ii) $\To$ (iii). Immediate.

(iii) $\To$ (i).
 First, we are going to prove that the assumption
that for every small direct summand $X$  there
is a unimodular element $t \in A$ such that
\begin{equation}
X \to t \text{ and } \f(t) \notin \str{X,t}
\end{equation}
leads to a contradiction.
Indeed, if $X \to t$
for a unimodular element $t \in A,$
(\theequation) implies existence of
a unimodular element $z$ with the property
$$
\str{X,t} \to z
$$
such that
$$
\f(t)=x+kt + mz
$$
for a suitable $x \in X,$ $k \in \Z$ and
a nonzero $m \in \Z;$ in effect, any basis of $X$ can then be extended
to a unimodular set of $A$ by adding to it the elements $t$
and $z.$
This, as it has been outlined above in Remark \ref{BrTrIndArg},
implies existence of a moiety $\cT$ of
a basis of $A$ such that
$$
\f(\str\cT ) \cap \str\cT=\{0\},
$$
which is impossible.

Thus we see that there exists a small direct summand
$X \le A$ such that
$$
X \to t \To \f(t) \in \str{X,t}
$$
for all unimodular elements $t \in A,$ or, equivalently,
$$
\f(t)=x(t) + k(t) t,
$$
where $x(t) \in X$ and $k(t) \in \Z,$ for all unimodular elements $t \in A$ satisfying $X \to t.$
However, quite similar to what we have seen in the proof of Proposition \ref{Nongenerators}, there is therefore
a large direct summand $M$ of $A,$ lying in a direct
complement of $X$ to $A,$ such that
$$
\f(t) = k(t) t
$$
for all unimodular elements $t \in M,$
which means
that either $\f|_M=\id_M,$ or $\f|_M=-\id_M,$ as required.
\end{proof}

Before moving on, observe that a particular element $\f \in \aut A$ can belong
to several subgroups $\LLambda k.$

\begin{Lem} \label{nongensnonrads}
Let $\f$ be an automorphism of the group $A$
which is neither an almost-radiation,
nor a normal generator of the group $\aut A.$ Then there is a unimodular set
$$
\{ x_i, y_i : i \in I\}
$$
of $A$ of cardinality $\vk$ such that for every $i \in I$
there are natural numbers $k_i,m_i$ satisfying
$$
\f x_i = k_i x_i + m_i y_i
$$
where $m_i \ge 2$ and $\f \in \LLambda {m_i}.$
\end{Lem}

\begin{proof} We claim that for every small
direct summand $X$ of $A$ there are unimodular
elements $t,z \in A$ such that
$$
\f t = kt+m z
$$
where $X \to t$ and $\str{X,t} \to z$
and $k,m$ are natural numbers such that
$m \ge 2$ and $\f \in \LambdamA.$
The result will then follow by
Remark \ref{BrTrIndArg}.

Let us prove the claim.
Indeed, since $\f$
is not a normal generator, there is,
by Claim 1 in the proof of Proposition \ref{Nongenerators},
a small direct summand $X_0$ such that
$$
X_0 \to t \To \str{X_0,t} \not\to \f(t)
$$
for all unimodular elements $t \in A.$
Next, given any small direct summand
$X,$ find a small direct summand $Y$ for which
we have
$$
X +X_0 \le Y.
$$
Let
$$
A = Y \oplus T_0.
$$
Set
$$
T_1=\f\inv(T_0) \cap T_0.
$$
Then for every unimodular element $t \in T_1,$
$$
\f(t) = k(t) t + m(t) z(t)
$$
where $k(t) \in \Z,$ $m(t) \in \N \setminus \{1\},$ and $z(t) \in T_0$
is a unimodular element of $A$ with $\str{Y,t} \to z(t).$
Now since $\f$ is not an almost-radiation,
the function $m(t)$ does not vanish on $T_1.$
Let then $m_0=m(t_0),$ where $t_0$ is a unimodular
element of $T_1,$ be the minimal positive
value of the function $m(t)$ on the set of unimodular
elements of $T_1.$
Clearly then,
$$
X \to t_0 \text{ and } \str{X,t_0} \to z(t_0),
$$
Finally, as we saw in the proof
of Claim 2 in the proof of
Proposition \ref{Nongenerators}, $\f \in \LLambda{m_0},$ as required.
\end{proof}

\begin{Lem} \label{normsubs_finite_tuples}
Let $N$ be a proper normal subgroup of the group $\aut A$
and let $\f_1,\f_2,\ldots,\f_n$ be elements
of $N.$ Then there is a natural number $m \ge 2$ such that
all $\f_1,\f_2,\ldots,\f_n$ are contained in $\LambdamA.$
\end{Lem}

\begin{proof} Let $n \ge 2.$ Understandably, none of $\f_i$ is a normal
generator of the group $\aut A,$ and we may
assume that none of $\f_i$ is an almost-radiation.
Suppose, towards a contradiction, that whenever
$$
\f_1 \in \LLambda{s_1}, \f_2 \in \LLambda{s_2},\ldots, \f_n \in \LLambda{s_n}
$$
where $s_i \ge 2,$ we have that
$$
\gcd(s_1,s_2,\ldots,s_n)=1.
$$
For simplicity's sake, we shall consider the case when $n=3,$
the general case being similar.
Consider a basis $\cB$
of $A$ partitioned into moieties as follows:
$$
\cB = \cX \sqcup \cY_1 \sqcup \cY_2 \sqcup \cY_3 \sqcup \cY_4.
$$
Let
$$
\cX =(x_i : i \in I) \text{ and } \cY_s=(y^s_i : i \in I), \qquad (s=1,2,3,4).
$$
By Lemma \ref{nongensnonrads}, the action of a suitable conjugate $\f_1'$ of $\f_1$
on $\cX$ is given by
$$
\f_1' x_i = k^1_i x_i + m^1_i  y^1_i
$$
where
$$
k^1_i,m^1_i \in \N,\quad \gcd(k^1_i,m^1_i)=1, \text{ and } \f_1 \in \LLambda{m^1_i}
$$
for all $i \in I.$ Observe that due to $\f_1'(\cX) \le \str{\cX,\cY_1},$
$\f_1'(\cX) \sqcup \cY_2$ is a unimodular set of $A,$ and then
there is a conjugate $\f_2'$ of $\f_2$ such that
$$
\f_2'( \f_1' x_i ) = k^2_i \f_1' x_i + m^2_i y^2_i, \qquad (i \in I),
$$
and due to $\f_2' \f_1' (\cX) \le \str{\cX,\cY_1,\cY_2}$
there is a conjugate $\f_3'$ of $\f_3$ such that
$$
\f_3' (\f_2' \f_1' x_i) = k^3_i \f_2' \f_1' x_i + m^3_i y^3_i, \qquad (i \in I),
$$
where
$$
k^s_i,m^s_i \in \N,\quad \gcd(k^s_i,m^s_i)=1, \quad \f_s \in \LLambda{m^s_i}
$$
for $s=2,3$ and for all $i \in I.$ Moreover, recall that we have
$$
\gcd(m^1_i,m^2_i,m^3_i)=1
$$
for all $i \in I.$ Written in detail, the action
of $\psi=\f_3' \f_2' \f_1'$ on $\cX$ is given by
$$
\psi x_i= k^3_i k^2_i k^1_i x_i + k^3_i k^2_i m^1_i y^1_i + k^3_i m^2_i y^2_i + m^3_i y^3_i
$$
where $i$ ranges over $I.$ However, since evidently
$$
\gcd( k^3_i k^2_i m^1_i, k^3_i m^2_i, m^3_i)=1, \qquad (i \in I),
$$
we obtain that
$$
\{x_i, \psi x_i : i \in I\}
$$
is a unimodular set of $A,$ and hence, by Proposition
\ref{Normalgenerators_suff_part},  $\psi$ is a normal generator
of the group $\aut A,$ which is impossible.

Consequently, there are natural numbers $s_1,s_2,s_3 \ge 2$
such that
$$
\f_1 \in \LLambda{s_1}, \f_2 \in \LLambda{s_2}, \f_3 \in \LLambda{s_3}
$$
satisfying
$$
\gcd(s_1,s_2,s_3) \ne 1,
$$
as required.
\end{proof}

\begin{Rem} \label{cY1-cYlp1}
\rm Generalizing somewhat the corresponding part of the proof of Lemma \ref{normsubs_finite_tuples},
one easily sees that the following statement is true. Suppose automorphisms $\f_1,\f_2,\ldots,\f_\ell$
of the group $A$ (or their suitable conjugates) all act on a certain unimodular set
$\{a_i,b_i : i \in I\}$ of $A$ of cardinality $\vk$ so that
$$
\f_s a_i = k_i^s a_i + m_i^s b_i,
$$
where $k_i^s,m_i^s \in \Z,$
$$
\gcd(k_i^s,m_i^s)=1 \text{ and } \f_s \in \LLambda{m_i^s}
$$
for all $i \in I$ and for all $s=1,2,\ldots,\ell.$ Then there is a
basis
$$
\cX \sqcup \cY_1 \sqcup \cY_2 \sqcup \ldots \sqcup \cY_\ell \sqcup \cY_{\ell+1}
$$
of the group $A,$ where
$$
\cX=(x_i : i \in I) \text{ and } \cY_s=(y_i^s :i \in I)
$$
for all $s=1,2,\ldots,\ell$ on which an appropriate product
$$
\psi=\f_\ell' \ldots \f_2' \f_1'
$$
of conjugates of $\f_\ell,\ldots,\f_2,\f_1,$ respectively,
acts so that
\begin{equation}
\psi x_i = \left( \prod_{s=1}^\ell k_i^s \right) x_i + \sum_{s=1}^{\ell-1} \left( \prod_{t=s+1}^\ell k_i^t\right) m_i^s y_i^s +m_i^\ell y_i^\ell.
\end{equation}
for all $i \in I.$ Observe also that for every $i \in I,$ the element in the right-hand
side of (\theequation) can be written in the form
$$
k_i x_i + m_i z_i
$$
where $k_i \in \Z,$ $z_i$ is a unimodular element of $A,$ and
\begin{equation} \label{gcd_in_prod_o_conjs}
m_i=\gcd(m_i^1,m_i^2,\ldots,m_i^\ell).
\end{equation}
Furthermore, $\{x_i,z_i : i \in I\}$ is a unimodular set of $A$ of
cardinality $\vk.$ $\triangle$
\end{Rem}

Let $\f$ be an automorphism of $A.$ Define
the set $\nu(\f)$ of prime numbers so that
$$
p \in \nu(\f) \iff \f \in \LLambda p.
$$
We will denote by $\mathbf P$ the set
of prime numbers of $\N,$ and, given a subset
$P \sle \mathbf P,$ we will denote by $\Omega_P$ the normal
subgroup $\bigcap_{p \in P} \LLambda p.$

\begin{Lem} \label{Lambda_cF}
Let $N$ be a proper normal subgroup of
the group $\aut A.$ Then

{\rm (i)} $\{\nu(\f) : \f \in N\}$ is
a centered family over $\mathbf P;$

{\rm (ii)} if $\cF$ is a filter over $\mathbf P,$
then the set
$$
\LLambda \cF = \bigcup_{P \in \cF} \Omega_P
$$
is a normal subgroup of $\aut A;$

{\rm (iii)} there exists an ultrafilter $\cF$ over
$\mathbf P$ such that $N \le \LLambda \cF.$
\end{Lem}

\begin{proof} (i) By Lemma \ref{normsubs_finite_tuples}.

(ii) Clearly, the set $\LLambda \cF$
is closed under taking inverses and under
conjugation by elements of $\aut A.$ Suppose
that $\f, \psi \in \LLambda \cF.$ Hence
$\f \in \Omega_P$ and $\psi \in \Omega_Q$
for some $P,Q \in \cF.$ But then $P \cap Q$
is also in $\cF,$ and so $\f,\psi$
are contained in the subgroup $\Omega_{P \cap Q},$
which implies that the product of $\f$
and $\psi$ is also contained in $\LLambda \cF.$

(iii) It is well-known any centered family
over a given set can be extended to a filter
over this set, and, in turn, any filter
is contained in a maximal filter (ultrafilter).
\end{proof}

In the case when $\cF$ is a principal ultrafilter
over $\mathbf P,$ the group $\LLambda \cF$
is of course equal to some subgroup
$\LLambda p$ for a suitable prime number $p.$
In the case when $\cF$ is a nonprincipal
ultrafilter over $\mathbf P,$ the situation
is no longer uniform. Indeed, on one hand,
Proposition \ref{uncount_conf_no_inf_descent} below
implies that $\LLambda \cF$ is equal to the
subgroup $\aR(A)$ of almost-radiations if
the rank of $A$ has {\it uncountable} cofinality.
On the other hand, we shall see below that in the case when
the rank of $A$ is of {countable} cofinality,
the group $\LLambda \cF$ is in fact a maximal
normal subgroup of $\aut A.$

\begin{Lem} \label{km-aut_in_LLs}
Let $s \ge 2$ be a natural number and let $\f \in \LLambda s.$
Suppose that there is a unimodular set $\{x_i,y_i : i \in I\}$
of cardinality $\rank(A)$ and integer numbers $k,m$ such that
$$
\f(x_i)=k x_i + m y_i
$$
for all $i \in I.$ Then $s$ is a multiple of $m.$
\end{Lem}

\begin{proof}
Assume that $M$ is a large direct summand
of $A$ and $l$ is an integer number such that
$$
\f(t) \equiv l t \Mod s
$$
for all $t \in M.$ Consider the moietous direct summand
$X$ generated by all $x_i.$ Then the intersection of
$X$ and $M$ is a nonzero direct summand, and so we can
find a unimodular element $x \in X \cap M$ such that
$$
\f(x) = k x + m y
$$
where $\{x,y\}$ is a unimodular set of $A.$ On the other hand, since $x \in M,$
$$
\f(x) = k x+m y= lx + sz
$$
for some $z \in A.$ Now as $\str{x,y}$ is a direct
summand of $A,$ the element $z$ must belong to
this summand, and the conclusion follows easily.
\end{proof}

\begin{Prop} \label{uncount_conf_no_inf_descent}
Let $\vk=\rank(A)$ have uncountable cofinality.
If $\f \in \aut A$ is neither normal
generator of the group $\aut A,$ nor an
almost-radiation, then there are at most finitely
many natural numbers $s$ such that
$\f \in \LLambda s.$
Consequently, if $(m_n)$ is any sequence of pairwise
distinct natural numbers $\ge 2,$ then
the intersection $\bigcap_n \LLambda {m_n}$
is equal to the group of almost-radiations.
\end{Prop}

\begin{proof} By Lemma \ref{nongensnonrads}, there
exists a unimodular set $\{x_i,y_i : i \in I\}$ of $A$
of cardinality $\vk$ such that
$$
\f x_i =k_i x_i + m_i y_i, \qquad (i \in I),
$$
where $k_i,m_i \in \N$ and $m_i \ge 2.$ As $\vk$ is of uncountable cofinality,
there are natural numbers $k,m,$ where $m \ge 2,$
and a subset $J$ of $I$ of cardinality $\vk$ satisfying
$$
\f x_j = k x_j + m y_j, \qquad (j \in J).
$$
Now, by Lemma \ref{km-aut_in_LLs}, whenever $\f \in \LLambda s,$  $s$ divides $m,$ whence the result.
\end{proof}

\begin{Rem} \label{in_inf_many_Lambdas}
\rm In the case when the cardinal $\vk=\rank(A)$ has {\it countable} cofinality,
for every infinite set $\{m_n : n \in \N\}$ of natural numbers, all greater
than $1,$ the intersection
$$
\bigcap_{n \in \N} \LLambda {m_n}
$$
contains automorphisms of $A$ which are not almost-radiations.
Indeed, let
$$
\vk=\sum_{n \in \N} \mu_n,
$$
where $\mu_n$ are cardinals $< \vk,$ and let
$$
\cB = \bigsqcup_{n \in \N} (\cX_n \sqcup \cY_n)
$$
be a basis of $A$ such that
$$
|\cX_n|=|\cY_n|=\mu_n
$$
for all $n \in \N.$ For each $n$ in $\N,$
pick up a bijection $f_n : \cX_n \to \cY_n.$
Define then the automorphism $\f \in \aut A$
by forcing it to act identically on each of the
sets $\cY_n$ and to act as
$$
\f x = x + m_0 m_1 \ldots m_n f_n(x), \qquad (x \in \cX_n)
$$
on each of the sets $\cX_n.$ Clearly, $\f \in \LLambda {m_0}$
and for every $n \ge 1,$ we have that $\f$ preserves
modulo $m_n$ all elements of the large direct
summand
$$
\bigoplus_{j \ge n} \str{\cX_j \sqcup \cY_j}
$$
which is of corank
$$
\sum_{j < n} (|\cX_j|+|\cY_j|)
$$
strictly less than $\vk,$ which means that $\f \in \LLambda {m_n}.$ $\triangle$
\end{Rem}

\begin{Lem} \label{Lambda_cF_count_conf}
Let $\vk=\rank(A)$ be of countable cofinality.

{\rm (i)} Suppose that $P$ is an infinite set of prime
numbers. Then there is an automorphism
$\f \in \aut A$ such that $\nu(\f)=P.$

{\rm (ii)} Let $\cF$ be a nonprincipal ultrafilter
over $\mathbf P.$ Then $\LLambda \cF$ is a maximal
normal subgroup of $\aut A,$ which is not
equal to any subgroup $\LLambda p,$
where $p$ is a prime number.

{\rm (iii)} Let $\cF_1,\cF_2$ be distinct
nonprincipal ultrafilters over $\mathbf P.$
Then $\LLambda {\cF_1} \ne \LLambda {\cF_2}.$
\end{Lem}

\begin{proof} (i) Apply the above construction to $P.$

(ii) Let $\f \not\in \Lambda_{\!A}(\cF).$ Then $\nu(\f) \not\in \cF$
and $\nu(\f)^c=\mathbf P \setminus \nu(\f)$ is necessarily infinite, for $\cF$
contains all cofinite subsets of $\mathbf P.$ By (i), there is
$\psi \in \Lambda_{\!A}(\cF)$ satisfying
$$
\nu(\psi) = \nu(\f)^c.
$$
Now, by Lemma \ref{normsubs_finite_tuples}, the normal
closure of $\{\f,\psi\}$ cannot be a proper normal subgroup
of $\aut A.$

To prove the second statement, take any prime
number $p$ and find $\s \in \Lambda_{\!A}(\cF)$
having the property that $\nu(\s)=P \setminus \{p\}.$
Clearly, $\s \notin \LLambda p.$

(iii) Since $\cF_1$ and $\cF_2$ are distinct,
there exists a moiety $P$ of $\mathbf P$
such that $P \in \cF_1$ and $P \notin \cF_2.$
Accordingly, by (i), there are $\f \in \LLambda {\cF_1}$
and $\psi \in \LLambda {\cF_2}$ satisfying
$$
\nu(\f)=P \text{ and } \nu(\psi)=P^c.
$$
Again, since the normal closure of $\{\f,\psi\}$
is equal to $\aut A,$ we get $\psi \not\in \LLambda {\cF_1}.$
\end{proof}

\begin{Th} \label{max_norm_subgrs}
Let $A$ be an infinitely generated free abelian group.

{\rm (i)} If the cardinal $\rank(A)$ is of uncountable cofinality,
then the maximal normal subgroups of the group $\aut A$
are exactly the normal subgroups $\LLambda p,$
where $p$ runs over the set $\mathbf P$ of prime
numbers of $\N.$

{\rm (ii)} If the rank of $A$ is of countable cofinality,
then the maximal normal subgroups of the group $\aut A$ are exactly
the normal subgroups
$$
\LLambda p \text{ and } \LLambda {\cF},
$$
where $p$ runs over $\mathbf P$ and $\cF$
over the family of all nonprincipal ultrafilters
over $\mathbf P.$
\end{Th}

\begin{proof}
By Lemma \ref{Lambda_cF} and Lemma \ref{Lambda_cF_count_conf}.
\end{proof}

\section{Normal generation of the principal
congruence subgroups}

Given a basis $\cB$ of the group $A,$ an automorphism $\f \in \aut A$ is said to be {\it $\cB$-moietous}
iff there is a moiety $\cC$ of $\cB$ such that $\f$
fixes all elements of $\cC$ and preserves the
subgroup $\str{\cB \setminus \cC}$ \cite{To_Berg}.
A {\it moietous} automorphism of the group $A$
is one that is $\cB$-moietous with regard
to a suitable basis $\cB$ of $A.$

\begin{Prop} \label{tau^m}
Let $m \ge 2$ be a natural number. Then the normal
closure of any automorphism of $A$ of type $\btau^m$
is equal to the $m$-th principal congruence subgroup $\GammamA$
of $\Gamma=\aut A.$
\end{Prop}

\begin{proof} Let $\cB=(x_i,y_i : i \in I)$ be a basis
of $A$ and let the action of $\tau \in \aut A$ of type $\btau$ on $\cB$
be given by
\begin{align*}
\tau x_i &= x_i+y_i,\\
\tau y_i &=y_i
\end{align*}
for all $i \in I.$ Next, we are going to describe
a family of automorphisms of $A$ that are natural
analogues of transvections in the groups $\mathrm{SL}_n(\Z).$
An automorphism $\beta \in \aut A$ is called
{\it block-unitriangular} with respect to $\cB$ \cite{BurnsPi} if
$$
\beta y_i=y_i
$$
and
$$
\beta x_i =x_i + z_i
$$
for all $i \in I,$ where $z_i$ is an element
of the subgroup $\str{y_i : i \in I}$ (in effect,
the `matrix' of $\beta$ with regard to the basis $\cB,$
formed by the row/column coordinate `vectors'
will be unitriangular and consist of
four distinct `blocks').

It can be verified that only minor and straightforward
modifications of part (a) of the proof of Theorem 2.2 given in \cite{BurnsPi} are
required to see that the following results are
true:
\begin{itemize}
\item any $\cB$-moietous automorphism
lying the group $\GammamA$ is a product
of suitable conjugates of $\cB$-block-unitriangular
automorphisms lying the group $\GammamA;$

\item any element of $\GammamA$
which acts identically on a moiety of $\cB$ is a product
of
a $\cB$-block-unitriangular automorphism
and
a $\cB$-moietous automorphism,
both of which lie in the group $\GammamA.$
\end{itemize}

Accordingly, we will be done if we demonstrate
that
\begin{itemize}
\item[(a)] any $\cB$-block-unitriangular
automorphism in $\GammamA$
is in the normal closure of $\tau^m;$

\item[(b)] for every $\f \in\GammamA$
there exists a moietous automorphism $\rho \in \GammamA$
such that $\rho\inv \f$ fixes pointwise
a moietous direct summand of $A.$
\end{itemize}

(a) As it has been established by Wans \cite{Wans}
(see also \cite[Cor. 2.9]{Meehan}), if $M$ is a
free module of infinite rank over a PID, then any
endomorphism of $M$ is a sum of three automorphisms
of $M.$

Write $X$ for the direct summand generated
by all $x_i$ and $Y$ for the direct
summand generated by all $y_i$  $(i \in I).$
Suppose that $\beta$ is a $\cB$-block-unitriangular
automorphism lying in the group $\GammamA.$
Then
\begin{align*}
\beta x_i &= x_i + mz_i, \qquad (i \in I),\\
\beta y_i &= y_i
\end{align*}
for suitable $z_i \in Y.$ Now,
due to the result by Wans, there are automorphisms
$\s_1,\s_2,\s_3 \in \aut A,$ all of which act identically
on $X$ and preserve $Y$ and such that
$$
\s_1 y_i + \s_2 y_i + \s_3 y_i =z_i, \qquad (i \in I).
$$
Letting $\s$ denote any of the automorphisms $\s_k,$
we see that
$$
\tau \s\inv x_i = x_i + y_i \To \s \tau \s\inv x_i = x_i +\s y_i
$$
and
$$
\s \tau \s\inv y_i = \s (\s\inv y_i)=y_i
$$
for all $i \in I.$ It then easily follows that
$$
\beta= \s_1 \tau^m \s_1\inv \cdot \s_2 \tau^m \s_2\inv \cdot \s_3 \tau^m \s_3\inv.
$$

(b) Take an arbitrary element $\f \in \GammamA.$
Suppose first that $\f$ is not an almost-radiation.
Then, by Lemma \ref{nongensnonrads}, there is a basis
$(a_i,b_i,c_i : i \in I)$ of $A$ such that
$$
\f a_i = k_i a_i + m_i b_i
$$
where $k_i,m_i$ are natural
numbers and $m_i \equiv 0\Mod m$
for all $i \in I.$ Partition the
unimodular set $\cC=\{c_i : i \in I\}$ into moieties,
$$
\cC = \cC_1\sqcup \cC_2
$$
and then consider an automorphism
$\rho_0$ of the group
$B=\str{ \{a_i,b_i : i \in I\} \cup \cC_1}$
such that
$$
\rho_0 a_i =\f a_i, \quad (i \in I),
$$
whose action on $B$ is isomorphic
to the action of $\f$ on $A.$ Set then
$$
\rho =\rho_0 \oplus \id
$$
over the direct sum
$$
A = B \oplus \str{\cC_2}.
$$
Then, by the construction, $\rho\inv \f a_i=a_i,$ where $i$ runs
over $I,$ as required.

To complete the proof, suppose that $\f \in \GammamA$ is an almost-radiation.
 If $m > 2,$ then $\f$ itself acts identically
on a moietous direct summand of $A.$ In the case when $m=2,$
$\f$ may act as $-\id_M$ on a large direct summand
$M$ of $A.$ If so, to construct $\rho$ in question,
choose a basis $\cC$ of $A$ which extends
a basis of $M,$ and force $\rho$
to send to the opposite every element
in a moiety of $\cC,$ and to fix each
element in the complementing moiety.
\end{proof}

Our next goal is to introduce a family of automorphisms
of $A$ whose normal closure contains a principal
congruence subgroup of level at least two.

\begin{Lem} \label{x2x+my}
Let $m,n \ge 2$ be natural numbers and let $\{x,y\}$ be a unimodular
pair of a free abelian group $B$ of rank at least $2(n-1)$ such that
$$
B =\str x \oplus C \text{ and } y \in C.
$$
Then there is an automorphism of $B$ of order $n$ which takes $x$ to $x+my$
and preserves $C.$
\end{Lem}

\begin{proof} The case when $n=2$ is easy. Suppose that $n \ge 3$ and let
$$
e_1,\ldots,e_{n-1},e_n,\ldots,e_{2n-2}
$$
be a basis of $B.$ Then the automorphism $\lambda \in \aut B$
such that
$$
\lambda(e_i)=
\begin{cases}
e_{i+1}             &\text{ if } i < n-1,\\
-e_1-\ldots-e_{n-1} &\text{ if } i=n-1,\\
e_i                 &\text{ if } i > n-1.
\end{cases}
$$
has order $n.$ Next, consider the automorphism $\sigma \in \aut B$
such that
$$
\sigma(e_i)=
\begin{cases}
-me_1+e_n             &\text{ if } i =1,\\
e_i + e_{i+n-1} &\text{ if } 1 < i \le n-1,\\
e_{i-n+1}                 &\text{ if } i > n-1.
\end{cases}
$$
To illustrate the construction, in the case when $n=3,$ the matrices
of $\lambda$ and $\sigma$ are
$$
[\lambda]=
\begin{pmatrix}
0	&	-1	&	0	&	0  \\
1	&	-1	&	0	&	0  \\
0	&	0	&	1	&	0  \\
0	&	0	&	0	&	1  \\
\end{pmatrix}
\text{ and }
[\sigma]=
\begin{pmatrix}
-m	&	0	&	1	&	0  \\
0	&	1	&	0	&	1  \\
1	&	0	&	0	&	0  \\
0	&	1	&	0	&	0  \\
\end{pmatrix},
$$
respectively.

We claim that
$$
\sigma\inv \lambda \sigma (e_1)=e_1 + me_n - me_{n+1}=e_1+m(e_n-e_{n+1}).
$$
Indeed, on one hand,
$$
\lambda \sigma(e_1)=\lambda (-me_1+e_n)=-me_2+e_n.
$$
On the other hand,
$$
\s(e_1 + me_n - me_{n+1})=-me_1+e_n +me_1-me_2=-me_2+e_n,
$$
Thus $\lambda \s (e_1)=\s(e_1 + me_n - me_{n+1}),$
which proves the claim.

Next, we prove that
$$
\sigma\inv \lambda \sigma \str{e_i : i > 1} =\str{e_i : i > 1},
$$
which is evidently equivalent to the fact that $\lambda$ stabilizes
the subgroup
$$
\sigma \str{e_i : i > 1}.
$$
Clearly,
\begin{align*}
\sigma  \str{e_i : i > 1} &=\sigma \str{e_2,\ldots,e_{n-1},e_n,\ldots,e_{2n-2}}  \\
			  &=\str{e_2+e_{n+1},\ldots,e_{n-1}+e_{2n-2},e_1,\ldots,e_{n-1}} \\
			  &=\str{e_i : i \ne n}
\end{align*}
and it is then obvious that $\lambda$ stabilizes the subgroup
$\sigma  \str{e_i : i > 1},$ as required.

Now, since $\{e_1,e_n-e_{n+1}\}$ is a unimodular
pair and $e_n-e_{n+1} \in \str{e_i : i > 1},$
the result follows easily.
\end{proof}


\begin{Lem} \label{ZaUshko}
Let $A=X \oplus Y$ be a decomposition of $A$ into a
direct sum of moietuous direct summands and let $(x_i : i \in I)$
and $(y_i : i \in I)$ be bases of $X,Y,$ respectively. Consider $\rho \in \aut A$
which stabilizes $X$ and fixes $Y$ pointwise. Then the normal
closure of $\rho$ contains the automorphism $\sigma$ which acts
identically on $X$ and such that
$$
\sigma(y_i)=y_i + x_i-\rho x_i
$$
for all $i \in I.$
\end{Lem}

\begin{proof} Consider the automorphism $\tau$ defined by
$$
\begin{cases}
\tau x_i =x_i +y_i,\\
\tau y_i=y_i
\end{cases}
$$
for all $i \in I.$ Suppose that
$$
\rho x_i = \sum \alpha_{ik} x_k
$$
for all $i \in I.$ Then
$$
\tau\inv( \rho x_i) = \sum \alpha_{ik} (x_k-y_k) =\sum \alpha_{ik} x_k -\sum \alpha_{ik} y_k
$$
for all $i \in I,$ which can be rewritten as
$$
\tau\inv \rho x_i = \rho x_i - \rho^\circ y_i,
$$
where $\rho^\circ$ is the automorphism of $Y,$
an isomorphic copy of the restriction of $\rho$ on $X.$

We have that
$$
\rho\inv \tau\inv \rho \tau x_i= \rho\inv \tau\inv (\rho x_i + y_i)
$$
whence, by the above argument,
$$
\rho\inv \tau\inv \rho \tau x_i= \rho\inv( \rho x_i - \rho^\circ y_i+y_i),
$$
or
$$
\rho\inv \tau\inv \rho \tau x_i =x_i-\rho^\circ y_i+y_i.
$$
for all $i \in I.$ Finally, letting $\pi$ denote the automorphism of
$A$ which interchanges $x_i,y_i$ for all $i \in I,$
we obtain that the automorphism $\sigma = \pi (\rho\inv \tau\inv \rho \tau) \pi$
acts identically on $X$ and
$$
\sigma y_i = y_i +x_i-\rho x_i
$$
for all $i \in I.$
\end{proof}

\begin{Lem} \label{gen_o_Km_rem_one_case}
Let $m \ge 2$ be a natural number and let an automorphism $\f \in \Gamma$ act on a unimodular
set $\{x_i, y_i : i \in I\}$ of $A$ of cardinality $\vk$
so that
$$
\f x_i = x_i +m y_i
$$
for all $i \in I.$ Then the normal closure of $\f$
contains the group $\GammamA.$
\end{Lem}

\begin{proof} Without loss of generality, we may
assume that the unimodular set $\{x_i, y_i : i \in I\}$
generates a moietous direct summand of $A.$ Choose
any basis $(x_i' : i \in I)$ of the direct
summand
$$
X=\str{x_i : i \in I}
$$
such that
$$
x_i \equiv x_i' \Mod m, \qquad (i \in I).
$$
In effect,
$$
x_i'=x_i+m x''_i
$$
for suitable elements $x''_i \in X,$ whence
$$
\f x_i = x_i+my_i = x_i+mx''_i + m(y_i-x''_i) =x_i'+my_i'
$$
for all $i \in I,$ and the set $\{x_i',y_i' : i \in I\}$
is clearly unimodular.

Now take any natural number $n \ge 2,$
and consider a basis of a direct complement
of the group $\str{x_i,y_i : i \in I}$
whose elements are presented in the form
$$
\{z_{i,k} : i \in I, k=1,\ldots,2n-2\}.
$$
Apply Lemma \ref{x2x+my} to construct an automorphism $\lambda \in \Gamma$
of order $n$ which acts on each of the subgroups
$$
\str{x_i'} \oplus \str{y_i', z_{i,1},\ldots,z_{i,2n-2}},
$$
as an automorphism of order $n$ and takes $x'_i + my_i'$ to $x'_i$
($i \in I$). We therefore see that
$$
\lambda \f x_i=x_i', \qquad (i \in I),
$$
that the element $\f_1=(\lambda \f)^n$ is in
the normal closure of $\f,$
and that the restriction
of $\f_1$ on $X$ can be, subject to the choice
of the basis $(x_i' : i \in I),$ any $n$-th power
of an element of the group $\Gamma_X(m).$
This enables us to assume that the action
of $\f_1$ on $X$ is as follows: let
$$
X=X_1 \oplus X_2
$$
be a decomposition of $X$ into a direct sum
of moietous direct summands; we then
require the action of $\f_1$ on $X_1$
be of isomorphism type of $\btau^{mn} \oplus \id$
on $A \oplus A$ and the action of $\f_1$
on $X_2$ be
equal to that one of the
identity map of $X_2.$

Let $V$ denote any direct complement of $X$ to $A.$ Clearly,
$$
A = X \oplus V = X \oplus \f_1 V,
$$
since $\f_1$ stabilizes $X.$ Our next goal is, keeping
in mind application of Lemma \ref{ZaUshko}, to find
an involution $\rho \in \aut A$ such that the element
$\f_2=\rho \f_1 \rho \f_1$ in the normal closure
of $\f$ preserves a moietous direct summand
and acts identically on its suitable direct
complement.

According to our assumptions,
 there is a basis
$(a_i,b_i,c_i : i \in I)$ of the group $X_1$
such that
\begin{align*}
\f_1 a_i &=a_i + mn b_i, \qquad (i \in I),\\
\f_1 b_i &=b_i, \\
\f_1 c_i &=c_i.
\end{align*}
We then define the action of the involution $\rho$
we have described above on $X_1$ as follows:
\begin{align*}
\rho a_i &=a_i, \qquad (i \in I),\\
\rho b_i &=c_i, \\
\rho c_i &=b_i.
\end{align*}
One then easily verifies that
\begin{align*}
\rho \f_1 \rho \f_1 a_i &=a_i+mn c_i +mn b_i, \qquad (i \in I),\\
\rho \f_1 \rho \f_1 b_i &=b_i, \\
\rho \f_1 \rho \f_1 c_i &=c_i.
\end{align*}
We then see that the restriction of $\f_2=\rho \f_1 \rho \f_1$ on $X_1$
is also of isomorphism type
$\btau^{mn} \oplus \id$
on $A \oplus A.$

Let
$
(v_i + x_{i,1} + x_{i,2} : i \in I)
$
be a basis of the group $\f_1 V,$
where $v_i \in V,$ $x_{i,1} \in X_1,$
$x_{i,2} \in X_2$ for all $i \in I.$ Complete the definition
of the involution $\rho$ by forcing it
to move each element of $X_2$ to the opposite one,
and such that
$$
\rho(v_i)=v_i+x_{i,2}
$$
for all $i \in I.$ It follows that
$$
\rho \f_1 V =\str{v_i + \rho(x_{i,1}) : i \in I} \sle V\oplus X_1.
$$
Evidently, we have as well that
$$
\rho \f_1 X_1 =X_1 \text{ and } \rho \f_1 X_2 =X_2
$$
and hence
$$
\f_2( V\oplus X_1) =\rho \f_1 \rho \f_1 ( V\oplus X_1) \sle V\oplus X_1.
$$
On the other hand, $\f_2 =\rho \f_1 \rho \f_1$ acts identically
on $X_2,$ as required.

Application of Lemma \ref{ZaUshko} to $\f_2,$ followed
by conjugation by a suitable automorphism,
demonstrates that the normal closure
of $\f$ contains an automorphism $\f_4$ whose
action on a certain basis $(e_i,f_i : i \in I)$ of $A$
is as follows:
\begin{alignat*} 2
\f_4 e_i &= e_i + mn f_i, & \qquad &(i \in I_1),\\
\f_4 e_i &= e_i + g_i,  &\qquad  &(i \in I_2),\\
\f_4 f_i &= f_i,        &        &(i \in I),
\end{alignat*}
where $I=I_1 \sqcup I_2$ is a partition
of $I$ into moieties and
$
g_i \in \str{f_i : i \in I}
$
for all $i \in I.$

Consider then a permutation $\pi$ of $I_1$
of order two which fixes no element of
$I_1$ and then the automorphism $\widetilde \pi$
of $A$ such that
\begin{alignat*} 2
\widetilde \pi e_i &=e_{\pi(i)}, & \qquad &(i \in I_1),\\
\widetilde \pi e_i &= e_i,  &\qquad  &(i \in I_2),\\
\widetilde \pi f_i &= f_i,        &        &(i \in I).
\end{alignat*}
It quickly follows that the automorphism
$$
\f_5 = \f_4 \widetilde \pi \f_4\inv \widetilde \pi
$$
in the normal closure of $\f$ fixes pointwise the
set
$$
\{e_i : i \in I_2\} \cup \{f_i : i \in I\}.
$$
Take $i \in I_1$ and write $i^*$ for $\pi(i).$
Then
\begin{alignat*} 3
\f_5 &e_i     &&= e_i     +  &&mn(f_i-f_{i^*})  \\
\f_5 &e_{i^*} &&= e_{i^*} +\, &&mn(f_{i^*}-f_{i}).
\end{alignat*}
Accordingly, $\f_5(e_i+e_{i*})=e_i+e_{i*}.$
To sum up the construction, $\f_5$
is of isomorphism type $\btau^{mn} \oplus \id$
on $A \oplus A.$ This easily implies that the normal closure
of $\f$ contains an automorphism of
type $\btau^{mn}.$

To complete the proof, take a pair of coprime
natural numbers $n_1,n_2 \ge 2.$ By the above argument,
there exists an automorphism $\tau$ of type $\btau$
such that
$$
\tau^{mn_1},\tau^{mn_2} \in \nc(\f),
$$
where $\nc(\f)$ is the normal closure of $\f.$ But then $\tau^m$ is in the normal
closure of $\f,$ and the result
follows by Proposition \ref{tau^m}.
\end{proof}

\begin{Prop} \label{norm_gens_of_Km}
Let $m \ge 2$ be a natural number and let $\f \in \Gamma$ have a unimodular set
$\{x_i,y_i : i \in I\}$ of $A$ of cardinality $\vk$ such that
$$
\f x_i = k_i x_i + my_i,
$$
where $k_i \in \Z,$ for all $i \in I.$ Then the normal
closure of $\f$ contains the group $\GammamA.$
\end{Prop}

\begin{proof} Clearly, we may assume, without
loss of generality, that all $k_i$ are natural
numbers $\le m.$ This enables us to assume that
$$
\f x_i = k x_i + my_i
$$
for all $i \in I$ for a suitable natural number
$k,$ which is coprime to $m.$ Let $\ell$ denote
the value of Euler's phi function at $m.$

By Remark \ref{cY1-cYlp1}, there is a basis
$$
\cX \sqcup \cY_1 \sqcup \ldots \sqcup \cY_\ell \sqcup \cY_{\ell+1}
$$
of the group $A,$ where
$$
\cX =(x_i : i \in I) \text{ and } \cY_s=(y_{i,s} : i \in I), \qquad (s=1,2,\ldots,\ell+1).
$$
and conjugates $\f_1,\f_2,\ldots,\f_\ell$ of $\f$ such that
their product $\psi = \f_\ell \ldots \f_2 \f_1$ acts on $\cX$
as follows:
$$
\psi x_i = k^\ell x_i + k^{\ell-1}m\, y_{i,1} + \ldots + km\,y_{i,\ell-1} + m\,y_{i,\ell},
\qquad (i \in I).
$$
Since $k^\ell \equiv 1 \Mod m,$ we see that there is a
unimodular set $\{x_i,z_i : i \in I\}$ of $A$ such that
$$
\psi x_i =x_i +mz_i
$$
for all $i \in I,$ and application of Lemma \ref{gen_o_Km_rem_one_case} completes
the proof.
\end{proof}

\begin{Prop} \label{ladder_rel_one_mor_gen}
Assume that for an automorphism $\f$ of $A$ there is a
maximal natural number $m \ge 2$ such that $\f \in \LambdamA.$
Then letting $N$ denote the normal closure $\mathop{\rm nc}(\f)$
of $\f,$ we have
$$
\GammamA \le N \le \LambdamA.
$$
\end{Prop}

\begin{proof} By the conditions, $\f$ is not an almost-radiation,
and hence, by Lemma \ref{nongensnonrads}, there exists
a unimodular set $\{x_i,y_i : i \in I\}$ of $A$ of cardinality $\vk$
such that
$$
\f x_i = k_i x_i+m_i y_i
$$
where $k_i,m_i \in \N$ and $\f \in \LLambda {m_i}$ for all $i \in I.$
Again by the conditions, $m_i \le m$ for all $i \in I,$
and so there is a natural number $m_0 \le m$ satisfying
$m_i=m_0$ for $\vk$ indices $i \in I.$ We claim
that $m_0=m.$ Indeed, we have that for a suitable $k \in \N,$
$$
\f x_j = k x_j + m_0 y_j
$$
for all elements $j$ of a subset $J$ of $I$ of cardinality $\vk,$
and then $m$ divides $m_0$ by Lemma \ref{km-aut_in_LLs}. Thus $m_0=m,$ as claimed.
The result then follows by Proposition \ref{norm_gens_of_Km}.
\end{proof}

\begin{Prop}[(Ladder Relation)] \label{LadderRel}
{\rm (i)} Let $N$ be a proper
normal subgroup of $\Gamma$ which is the normal closure
of finitely many elements of $\Gamma.$
Suppose there exists a maximal natural number $m \ge 2$ satisfying $N \le \LambdamA.$
Then
$$
\GammamA \le N \le \LambdamA.
$$

{\rm (ii)} Suppose that the cardinal $\rank(A)$ is of uncountable
cofinality. If $N$ is a proper normal subgroup of $\Gamma$
which is not contained in the group of all almost-radiations,
then
$$
\GammamA \le N \le \LambdamA
$$
for a suitable natural number $m \ge 2.$
Accordingly, every normal subgroup of the
group $\aut A$ satisfies Brenner's ladder
relation.
\end{Prop}

\begin{proof} (i) Suppose that $N$ is the normal
closure of automorphisms $\f_1,\f_2,\ldots,\f_n$
from $\Gamma.$ Due to the maximality condition
in the hypothesis, there exist $\f_i$ which
are not almost-radiations, and indeed we
may assume that all $\f_i$ are not almost-radiations.
Since the case when $n=1$
is covered by Proposition \ref{ladder_rel_one_mor_gen},
assume that $n \ge 2.$
Applying the maximality condition on $m$
once again, we see that
\begin{equation}
\f_1 \in \LLambda {s_1}, \f_2 \in \LLambda {s_2},\ldots, \f_n \in \LLambda {s_n}
\To
\gcd(s_1,s_2,\ldots,s_n) \le m
\end{equation}
for all $s_1,s_2,\ldots,s_n \ge 2.$

By Remark \ref{cY1-cYlp1}, for every $i=1,2,\ldots,n,$ there is a conjugate
$\f_i'$ of $\f_i$ such that the automorphism
$\psi=\f_1' \f_2' \ldots \f_n'$ acts on a suitable
unimodular set $\{x_i,z_i : i \in I\}$ of $A$ of cardinality $\vk$ so that
$$
\psi x_i =k_i x_i + m_i z_i, \qquad (i \in I),
$$
where $k_i$ and $m_i \le m$ are natural numbers
(the last inequality is due to (\theequation) and \eqref{gcd_in_prod_o_conjs} in Remark \ref{cY1-cYlp1}).
Consequently, there exists a natural number $m_0 \le m$
such that for an appropriate $k \in \N$ and for an appropriate subset $J$ of $I$
of cardinality $\vk,$
$$
\psi x_j = k x_j + m_0 z_j
$$
for all $j \in J.$
According to Lemma \ref{km-aut_in_LLs},
$m$ divides $m_0,$ and hence $m_0=m.$ Then,
by Proposition \ref{norm_gens_of_Km},
$$
\GGamma m=\GGamma{m_0} \le N \le \LambdamA,
$$
and we are done.

(ii) Given a finite tuple $\f_1,\f_2,\ldots,\f_n$
of elements of $N,$ set $\mu(\f_1,\f_2,\ldots,\f_n)$ to be the
maximal natural number $m$ such that
$$
\mathop{\rm nc}(\f_1,\f_2,\ldots,\f_n) \le\LambdamA,
$$
or $\infty,$ otherwise. Since $N$ contains
automorphisms of $A$ which are not almost-radiations,
we have, by Proposition \ref{uncount_conf_no_inf_descent},
that the function $\mu$ takes finite values
at some finite tuples of $N.$ Choose then a finite tuple $\psi_1,\psi_2,\ldots,\psi_k$
of $N$ at which the value of $\mu$ is minimal.
Consequently,
$$
m_0=\mu(\psi_1,\psi_2,\ldots,\psi_k) \le \mu(\psi_1,\psi_2,\ldots,\psi_k,\f) \le \mu(\psi_1,\psi_2,\ldots,\psi_k)
$$
for all $\f \in N.$
Thus
$N \le \LLambda {m_0}.$
Applying (i) to the tuple $\psi_1,\psi_2,\ldots,\psi_k,$ we see that
$$
\GGamma{m_0} \le N \le \LLambda{m_0}.
$$

Now in order to show that an arbitrary normal subgroup $K$ of
the group $\aut A$ satisfies Brenner's ladder
relation, observe that if $K=\aut A,$ then
$$
\GGamma 1 =K= \LLambda 1=\aut A,
$$
and that if all elements of $K$ are almost-radiations,
then
$$
\{\id_A\}=\GGamma 0 \le K \le \LLambda 0=\aR(A).
$$
\end{proof}

It is not hard to show that part (ii) of
Proposition \ref{LadderRel} does not
hold in the case when the cardinal $\rank(A)$
is of countable cofinality. Indeed,
use Remark \ref{in_inf_many_Lambdas} to construct
for every prime number $p > 2$ an automorphism
$\f_p$ of $\aut A,$ which is not an almost-radiation, determined by the set $\mathbf P \setminus \{p\}$
of prime numbers. Set then
$$
N =\mathrm{nc}( \f_p : p \in \mathbf P \setminus \{2\}).
$$
By the construction, $N \le \LLambda 2,$
and $m=2$ is the maximal natural number
satisfying $N \le \LambdamA.$
Now if an automorphism $\rho$ of type $\btau^2$
is contained in $N,$ we get that
$$
\GGamma 2 \le \mathrm{nc}(\f_{p_1},\ldots,\f_{p_s})
$$
for suitable primes $p_i > 2.$ However,
if $q > 2$ is a prime number which exceeds
all the primes $p_i,$ we obtain that
$$
\GGamma 2 \le \nc(\f_{p_1},\ldots,\f_{p_s}) \le \LLambda q,
$$
which is impossible.

\end{document}